\documentclass{article}
\usepackage{amsfonts,amsthm}
\usepackage{amscd,amsmath,amssymb}
\usepackage{amssymb}

\newtheorem{definition}{Definition}
\newtheorem{proposition}{Proposition}
\newtheorem{theorem}{Theorem}
\newtheorem{corollary}{Corollary}
\newtheorem{lemma}{Lemma}

\theoremstyle{definition}

\newcommand{\tr}{\operatorname{tr}}
\newcommand{\diag}{\operatorname{diag}}

\def\leq{\leqslant}
\def\geq{\geqslant}

\renewcommand{\Re}{\mathop{\mathrm{Re}}\nolimits}

\begin{document}

\title{Nonunitary representations of the groups of $U(p,q)$-currents for $q\geq p>1$}

\author{A.~M.~Vershik\footnote{St.~Petersburg Department of Steklov Institute of Mathematics; St.~Petersburg State University, St.~Petersburg, Russia; Institute for Information Transmission Problems, Moscow, Russia. Email: {\tt avershik@gmail.com}. Supported by the RSF grant 17-71-20153.} \and $\boxed{\text{M.~I.~Graev}}$}

\date{November 24, 2017.}

\maketitle

\begin{abstract}
The purpose of this paper is to give a construction of representations
of the group of currents for semisimple groups of rank greater than
one. Such groups have no unitary representations in the Fock space,
since the semisimple groups of this form have no nontrivial cohomology
in faithful irreducible representations. Thus we first construct
cohomology of the semisimple groups in nonunitary representations. The
principal method is to reduce all constructions to Iwasawa subgroups
(solvable subgroups of the semisimple groups), with subsequent
extension to the original group. The resulting representation is
realized in the so-called quasi-Poisson Hilbert space associated with
natural measures on infinite-dimensional spaces.

\end{abstract}

Mark Iosifovich Graev died in Moscow on April 22, 2017 (born on November 22, 1922). He belonged to a rather narrow circle of the best mathematicians of Russia whose creative activity fell on the second half of the 20th century. The beginning of his mathematical biography from the mid 1940s is entwined with the Moscow school of algebra headed by A.~G.~Kurosh, who requested M.~I.\ to continue the research on the theory of free continuous groups initiated by A.~A.~Markov in the 1930s. This theory was the subject of M.~I.'s PhD thesis and his first papers, which became widely known in the algebraic literature and won the first prize of the Moscow Mathematical Society (1948). From the late 1940s, M.~I.\ participates in I.~M.~Gelfand's seminar and gradually becomes a regular collaborator and the principal coauthor of I.~M.\ in classical representation theory. Their work (along with the previous work by I.~M.~Gelfand and M.~A.~Naimark) contains a huge body of various results in one of the main mathematical theories of the 20th century. I.~M.\ and M.~I.\ wrote together nearly a hundred papers and three monographs, including two volumes of the series {\it Generalized Functions}. Besides, also M.~I.\ wrote two monographs on integral geometry and hypergeometric functions. In the last 10 years, M.~I.\ and I have worked  to continue the series of papers  initiated together with I.~M.~Gelfand as early as in the 1970s. This article is our last joint work. Speaking about the characteristic features of M.~I.\ as a mathematician, I would  mention his unparalleled  thoroughness in daily work and his rare devotion to science, to which he gave his life unreservedly. Great achievements of M.~I.\ in mathematics are in no small part due to remarkable features of his character: kindness, tranquility, modesty, unselfishness. I hope that his name will always serve as a remarkable example of scientist for future generations of mathematicians.

See also an essay about M.~I.\ in {\it Russian Mathematical Surveys} {\bf 63}, No.~1(379), 169--182 (2008), and the list of publications of M.~I.\ therein.


\section{Introduction}

This paper, which we conceived several years ago, but discussed in detail with M.~I.\  only in the most recent time, contains a project of constructing nonunitary representations of groups of currents with coefficients in semisimple groups of  the
 type~$U(p,q)$. It is worth explaining what is the essence and necessity of such a project. Groups of currents, or functional Lie groups, are groups of functions on manifolds (or groups of sections of fiber bundles over manifolds with groups as fibers)
 with values in a finite-dimensional Lie group (or even an arbitrary locally compact group), called the group of coefficients. Since the 1970s, many papers have been written on representations of such groups, among which we mention only the most important ones:
 \cite{Ar,VGG1,VGG2,VGG3,V3,Part,Gui,Del,Ism,AV}.

It was understood that in order to construct unitary representations of a group of currents, one need to use nontrivial cohomology of the group of coefficients with values in irreducible unitary representations of this group. Representations in which the 1-cohomology is nontrivial are called special, and a necessary condition for a representation to be special is that it is ``glued'' to (inseparable from) the identity representation. Of course, the main interest was in semisimple groups; their study from this viewpoint was initiated in~\cite{VGG1}.

Special irreducible representations exist by no means for all Lie groups. For example, in the class of real forms of semisimple Lie groups, this is the case only for the groups of real rank~1 except for the symplectic groups, that is, only for those groups that do not have Kazhdan's property (=~for which the identity representation is isolated in the Fell topology in the space of irreducible unitary representations). In other words, these are the groups $U(p,1), O(p,1)$, $p=1,2, \dots $. For such groups, the corresponding theory was constructed in
 \cite{VGG1,VGG2} and other papers in the 1970s--1980s. After some break, this research was continued in our subsequent papers; here is the complete list of them: \cite{VG1,VG2,VG3,VG4,VG5,VG6,VG7,VG8,VG9,VG10,VG11,
  VG12,VG13}. In particular, a new method, outlined in \cite{VGG3}, was developed to realize these representations,  which instead of the Fock space uses more natural probabilistic constructions: the integral, Poisson, and quasi-Poisson models. The main idea, common for all papers of this series, is that we first develop a representation theory for a solvable subgroup of the semisimple group, namely, for the Iwasawa group, and then extend it to the whole semisimple group. The naturalness of this idea, in any case for real forms of the type $U(p,q),O(p,q),Sp(p,q)$, follows from the structure of the group itself and its Iwasawa subgroup. However, passing to solvable subgroups allows one to use other models of factorizations different from the Fock factorization and more convenient. We mean that the class of L\'evy processes and, in particular, the gamma process provides alternative models of Hilbert spaces and factorizations convenient for representation theory and quantization. One of the important consequences of this approach is  the idea of an analog
  of the Lebesgue measure in an infinite-dimensional space, which is a renormalization of the distribution of the gamma process, has a large group of symmetries, and is related to the Poisson--Dirichlet measures~\cite{V1,V2}.

The study of these groups is not yet complete, however, one should turn to cases that do not fall into the framework of the existing schemes. Namely, what should one do with representations of groups of currents with values in  $O(p,q),U(p,q)$ for $q>1$ and $Sp(p,q)$, as well as in other real forms of semisimple groups?

Since these groups have no special irreducible unitary representations, one should turn to special irreducible  nonunitary representations, and this is what we do in our last papers. In this paper, we complete such a construction for the groups~$U(p,q)$. The plan is as follows: first, we construct a cocycle in a nonunitary representation of the Iwasawa subgroup of~$U(p,q)$. Accordingly, we describe in detail the structure theory of Iwasawa groups, which is not very popular. The Iwasawa subgroup of
 $U(p,q)$ is a solvable group which is the semidirect product of a nilpotent subgroup of lower triangular matrices~$S$ and the so-called Heisenberg group~$N_{p,q}$. While nilpotent groups have no nonidentity special unitary representations, the situation with solvable groups is quite unclear. We do not know whether Iwasawa groups of rank greater than one have special unitary representations. But anyway these special representations cannot be extended to special unitary representations of the whole semisimple group. Hence we define natural nonunitary special representations of the Iwasawa group, i.e., construct nontrivial 1-cohomology in a nonunitary representation of this group, and then extend the representation and the cocycle to the whole group~$U(p,q)$. This can be done only if we give up unitarity, but preserve the boundedness of the representation operators for most elements of the group, except for the main involution, which can be unbounded. After having extended the representation and the cocycle, we proceed to reproduce the construction of the quasi-Poisson model of a representation of the group of its currents. This presents no great difficulties. The construction is sketched in broad strokes and will be considered in detail later. It is important that the constructed representation acts in a Hilbert space (the $L^2$ space with respect to a quasi-Poisson measure) by operators that are either bounded (for most elements of the group) or densely defined (for the involution and some elements of a compact subgroup), and all essential properties of representations of groups of currents that hold for groups of rank~1 are preserved.

It is yet to be explored what vertex-like operators can be defined in the spaces of these representations and what relation can these models have to the corresponding Fock space theory etc. But I believe that this paper lays the groundwork for subsequent considerations.

\section{The group $U(p,q)$ and its solvable subgroup (Iwasawa group)}

\subsection{The group $U(p,q)$}
By definition, $U(p,q)$ is the group of linear transformations of the space $\mathbb {C}^{p+q}$ preserving a Hermitian form of signature~$(p,q)$. Here we chose this form so that the maximal solvable subgroup of $U(p,q)$ takes the simplest form:
$$
\sum\limits_{i=1}^p (x_i \bar x_{q-p+i}+\bar x_i  x_{q-p+i}) +\sum\limits_{i=1}^q |x_{p+i}|^2, \quad q\geq p.
$$

All complex matrices of order $p+q$ will be written in block form:
$$
g=\begin{pmatrix}{}
        g_{11}&g_{12}&g_{13}\\
        g_{21}&g_{22}&g_{23}\\
        g_{31}&g_{32}&g_{33}
\end{pmatrix},
$$
where the diagonal blocks are matrices of orders $p$, $q-p$, and $p$, respectively. We assume that
 $q\geq p$. In particular, if $q=p$, then these matrices degenerate into $2\times2$ block matrices.

In this notation, $U(p,q)$ is defined as the group of block matrices satisfying the condition
\begin{equation}{}\label{1-2}
g \sigma   g^*= \sigma  ,\quad \mbox{where}\quad
\sigma   =\left(\begin{array}{ccc}0&0&e_p\\ 0&e_{q-p}&0\\ e_p&0&0\end{array}\right)
\end{equation}
is the involution and the symbol $*$ stands for the composition of complex conjugation and matrix transpose.

This condition is equivalent to the following relations between the blocks of~$g$:
\begin{equation}\label{1-3}
\begin{gathered}{}
g_{13}g_{31}^*+g_{12}g_{32}^*+g_{11}g_{33}^* =e_p,
\\
g_{23}g_{21}^*+g_{22}g_{22}^*+g_{21}g_{23}^*=  e_q,
\\
g_{13}g_{21}^*+g_{12}g_{22}^*+g_{11}g_{23}^*=0,
\\
g_{13}g_{11}^*+g_{12}g_{12}^*+g_{11}g_{13}^*=0,
\\
g_{23}g_{31}^*+g_{22}g_{32}^*+g_{21}g_{33}^*=0,
\\
g_{33}g_{31}^*+g_{32}g_{32}^*+g_{31}g_{33}^*=0.
\end{gathered}
\end{equation}

The real dimension of the group $U(p,q)$ is  $(p+q)^2$.

\subsection{The subgroups $S$ and $N$}

Denote by $S$ the subgroup of block diagonal matrices from  $U(p,q)$ of the form
$$
\begin{pmatrix}{}
        s^{-*}&0&0\\
        0&e_{q-p}&0\\
        0&0&s
\end{pmatrix},
$$
where  $s$ ranges over the subgroup of complex lower triangular matrices $\|r_{ij}\|$ of order~$p$ such that $r_{ij}=0$ for $i<j$ and the diagonal entries~$r_{ii}$  are real and positive.

Obviously, $S$ is a solvable group of rank~$p$; its real dimension is equal to~$p^2$.

Denote by $N$ the (Heisenberg) subgroup of block matrices of the form
$$
\begin{pmatrix}{}
        e_p&0&0\\
        -z^*&e_{q-p}&0\\
        \zeta&z&e_p
\end{pmatrix},
$$
where $\zeta$ is a complex $p \times p$ matrix, $z$ is a complex $p\times (q-p)$ matrix, and the condition for these matrices to belong to the group $U(p,q)$ has the following form:
$$
\zeta + \zeta^* - zz^*=0, \quad \mbox{i.e.}, \quad \zeta = n-\frac12 zz^*,
$$
where $n$ is a skew-Hermitian matrix. In other words, $\zeta$ has a given real part determined by the matrix~$z$.

Elements of~$N$ will be written as pairs $(\zeta,z)$ with the multiplication law
$$
(\zeta_1,z_1)(\zeta_2,z_2)=(\zeta_1+\zeta_2-zz^*,z_1+z_2).
$$

Obviously, $N$ is a nilpotent subgroup of $U(p,q)$ of real dimension $p(2q-p)$.

Denote by  $Q$ the subgroup in $U(p,q)$ generated by~$S$ and~$N$. It is the semidirect product of these groups:
$$Q=S\rightthreetimes N.$$
Elements of~$S$ act on~$N$ as group automorphisms:
$$s:(\zeta,z) \rightarrow (\zeta,z)^s=(s^*,sz).$$

\subsection{The Iwasawa subgroup $P$}

For every semisimple Lie group~$G$, the following analytic Iwasawa decomposition holds:
$$G=NTK,$$
where $N$ is a maximal nilpotent subgroup, $T$ is $(\mathbb{R}_+^*)^p$ with $p$ the rank of~$G$, and $K$ is a maximal compact subgroup in~$G$. Of course, this is a direct product of spaces, but not a direct product of groups,
however, the first two components form a subgroup in~$G$, namely, the semidirect product of~$T$ and~$N$.

\begin{definition}
The semidirect product $P=T\rightthreetimes N$ of the groups~$N$ and~$T$ is called the Iwasawa subgroup of the semisimple group~$U(p,q)$. The subgroup~$N$ is called the Heisenberg group (see below). In the case
$p=q$, the group~$N$ is the additive (commutative) group of complex matrices of order~$p$, i.e., the ``degenerate'' Heisenberg group.
\end{definition}

In the matrix realization of the group $U(p,q)$ adopted here, we may assume that $N$ is the subgroup of all block matrices in which the entry~$g_{33}$ ranges over the subgroup of lower triangular nilpotent matrices of order~$p$ and $T$ is the intersection of the subgroup~$S$ with the group of block diagonal matrices.

Thus we have the analytic decomposition $U(p,q)=PK$  (of spaces, but not groups),  which plays an important role in our constructions.

It follows from the definition that the group~$P$ in our case is isomorphic to the semidirect product $Q=S \rightthreetimes N$ introduced above, i.e., to the subgroup of all block triangular matrices of the form
$$
\begin{pmatrix}{}
       (s^*)^{-1}&0&0\\
       -(sz)^*&e_{q-p}&0\\
	s\zeta&sz&s
\end{pmatrix}.
$$

Let us give a direct description of the Heisenberg group~$N=N_{p,q}$. Denote by
$A\sim {\mathbb R}^{p^2}$  the additive group of skew-Hermitian matrices~$n$ of order~$p$, by $\mathcal A$, the dual group of unitary characters~$\chi(n)$, and by $Z\sim {\mathbb C}^{p(q-p)}$, the additive group of
complex $p\times(q-p)$ matrices.

In this notation, the Heisenberg group~$N$ can be presented as the group of pairs $(n,z)$, $n\in A$, $z\in Z$, with the multiplication law
$$
(n_1,z_1)(n_2,z_2)=(n_1+n_2-\frac12(z_1z_2^*-z_2 z_1^*),z_1+z_2).
$$
In other words, $N$ is the central extension of the commutative group~$Z$ by the commutative group~$A$ and the symplectic 2-form
$\omega:Z\times Z\rightarrow A$, where $\omega(z_1,z_2)=z_1z_2^*-z_2z_1^*$. The center of~$N$ is the subgroup~$A$, and the quotient by the center is the group~$Z$. Recall that the standard realization of the classical three-dimensional Heisenberg group as the sum
 ${\mathbb C} +\mathbb R$ with the symplectic form
$\omega(u_1,u_2)=i[u_1\cdot {u_2}^* - u_2\cdot {u_1}^*]\in \mathbb R$, $u_1,u_2 \in \mathbb C$, in our notation corresponds to
$n\in i\mathbb R$, $z=u\in \mathbb C$, i.e., the classical real three-dimensional Heisenberg group corresponds to the case
$p=1$, $q=2$. For $p=1$ and arbitrary~$q$, we obtain the ordinary $(2q+1)$-dimensional Heisenberg group; such groups will be called the one-row Heisenberg groups.

The real dimension of the group $N_{p,q}$ (for $p\ne q$) is  $2p(q-p)+p^2=p(2q-p)$, and the dimension of the Iwasawa group (subgroup in  $U(p,q)$) is~$2pq$. The dimension of the maximal compact subgroup
 ($U(p)\times U(q)$) is~$p^2+q^2$.

We emphasize that the study of the Iwasawa group as a solvable group, its representations and cohomology deserves attention in its own right, regardless of applications to representations of groups of currents considered below and to other problems.

\section{Nondegenerate irreducible unitary representations of the Heisenberg group~$N$ and Iwasawa group~$P$}

In what follows, we will speak of the Iwasawa group rather than the Iwasawa subgroup, since it is of interest for us in itself. To begin with, the groups~$N$ and~$Z$ described above are acted on by the group~$S$ of automorphisms:
$$
n\rightarrow sns^*, \quad \chi(n) \rightarrow \chi(sns^*), \quad z \rightarrow sz.
$$

Thus, the set of elements of the group~$A$ and the set of elements of the group~$\mathcal A$ of its unitary characters split into orbits of the action of~$S$.

\begin{definition}
Let us say that an element of either of these groups, as well as its $S$-orbit, is nondegenerate if the maps $n\rightarrow sns^*$ and $\chi(n)\rightarrow \chi(sns^*)$ are faithful.

Elements of the same nondegenerate $S$-orbit will be called conjugate.
\end{definition}

In particular, characters of the from
$$
\chi^\epsilon (n) = \exp(\tr(\epsilon n)), \quad\text{where}\quad \epsilon=\diag(\epsilon_1,\dots,\epsilon_p), \quad \epsilon_i=\pm 1,
$$
are nondegenerate. The orbit of every character~$\chi^\epsilon$ consists of the characters of the form
$$
\chi_s^\epsilon (n)=\chi^\epsilon(sns^*).
$$

\begin{theorem}
The set of all nondegenerate characters is exhausted by the characters of the form $\chi_s^\epsilon$. Thus there are exactly
$2^p$ nondegenerate $S$-orbits of characters.
\end{theorem}

In this section, with each nondegenerate character we will associate an irreducible unitary representation of the Heisenberg group.

%


Every character $\chi$ gives rise to a unitary representation of the group~$N$ in the Hilbert space $H_\chi=L^2(Z,dz)$, where $dz$ is the Lebesgue measure on~$Z$.

The operators $T$ of this representation are given by the following formula:
$$T(n^0,z^0)f(z)=\chi^\epsilon(n-\frac12(zz^{0*}-z^0z^*))f(z+z^0).$$
Note that the space~$H$ of the regular representation of the group~$N$ is the direct integral $H=\int\limits^\otimes H_\chi d\chi$, where $d \chi$ is the invariant measure on the group of characters $\chi$.

Let us study the structure of the spaces~$H_\chi$ associated with nondegenerate characters~$\chi$ and their decomposition into irreducible subspaces.

\subsection{The representation~$T$ of the group~$N$ induced by a nondegenerate character}

Usually, one realizes the representation of the Heisenberg group in the Hilbert space~$H_\chi$ as a representation induced by the characters~$\chi^\epsilon$. According to the general definition of the induced representation $H_\chi=L^2(Z,dz)$, the representation operators corresponding to elements of the group~$N$ in the space~$H_\chi$ are given by the formula
\begin{equation}
T^\epsilon(n^0,z^0)f(z)=\chi^\epsilon  (n-\frac12 (z z^{0*}-z^0z^*)) f(z+z^0),
\end{equation}
or, in more detail,
\begin{equation}
T^\epsilon(n^0,z^0)f(z)=\exp\bigg(\sum\limits_{i=1}^p (\epsilon_i n_{ii}-\frac12(z_iz_i^{0*}-z_i^0 z_i^*))\bigg) f(zz^0),
\end{equation}
where $z_i$ is the $i$th row of the matrix~$z$.

First, consider the case of a nondenenerate character on~$N$ of the form
 $\chi^\epsilon\!(n)\! = \!\exp(\tr \epsilon n)$,
where $\epsilon\!=\!\diag (\epsilon_1,\ldots,\epsilon_p)$, $\epsilon_i\!=\!\pm 1$.

Note that each of the nondegenerate $S$-orbits in the space of characters has a character~$\chi^\epsilon$ as a representative.

However, we will need another realization of the representation~$T$, which is obtained by replacing functions~$f(z)$ with $f(z)e^{-\frac12\tr zz^*}$.

 \begin{theorem}
In the new realization, the representation~$T$ of the group~$N$ acts in the Hilbert space $L^2(Z,d\mu(z))$, where $d\mu(z)=e^{-\tr zz^*}dz$ is the Gaussian measure on $Z$, as follows:
 {\small
\begin{equation}
T^\epsilon (n^0,z^0) f(z) \!= \!\exp \bigg( \sum\limits_{i=1}^p (\epsilon_i n_{ii}-\frac12 z^0z^{0*}\! - \!\sum\limits_{i;\epsilon_i=1} z_iz_i^{0*}\!-\!\sum\limits_{i;\epsilon_i=1}z_i^0z_i^*)\bigg)f(z+z^0),
\end{equation}}
where $z_i$ is the $i$th row of the matrix~$z$.
\end{theorem}

Indeed, in the new realization, the formula for the action of~$T$ is obtained from the formula in the old realization by adding the following factor:
\begin{equation}
e^{-\tr[(z+z^0)]z^*z^{0*}+\tr zz^{0*}}=\chi\bigg(\frac 12 \sum\limits_{i=1}^p(z_i^0z_i^{0*}+z_iz_i^{0*}+z_i^0z_i^*)\bigg).
\end{equation}

\begin{corollary} The Hilbert space $L^2(Z,\mu)$ can be written as the tensor product
\begin{equation}
L^2(Z,\mu)=\otimes_{i=1}^p L^2(Z_i,\mu_i)
\end{equation}
of the Hilbert spaces $L^2(Z_i,\mu_i)$, where $Z_i$ is the $i$th row of the matrix~$Z$ and $\mu_i$ is the Gaussian measure on~$Z_i$,  acted on by the representations~$T_i$ of the one-row Heisenberg groups with elements $(n_i,z_i)$:
\begin{equation}
T_i^\epsilon(n_i^0,z_i^0)f(z_i)=\exp(n_{ii}-\frac12 z_i^0z_i^{0*}-z_iz_i^{0*})f(z_i+z_i^0)\quad\mbox{ for } \epsilon_i=1,
\end{equation}
\begin{equation}
T_i^\epsilon(n_i^0,z_i^0)f(z_i)=\exp(-n_{ii}-\frac12 z_i^0z_i^{0*}-z_i^0z_i^*)f(z_i+z_i^0)\quad\mbox{ for }\epsilon_i=-1.
\end{equation}
\end{corollary}

\subsection{The irreducible decomposition of the representation of the group~$N$
in the space $H_\chi^\epsilon$}

\begin{definition}
The subspace of functions $f(z)$ from $L^2(Z,\mu)$ that are holomorphic with respect to the rows~$z_i$ where $\epsilon_i=1$, and anti-holomorphic with respect to the rows~$z_i$ where $\epsilon_i=-1$ will be called
the Bargmann space of type~$\epsilon$ and denoted by $K^\epsilon$.
\end{definition}

It follows from Theorem~{\rm2} and its corollary that this subspace is closed and invariant. The representation operators corresponding to elements of the group~$N$ are defined on this space by the formulas given above.

It is also obvious that the Bargmann space~$K^\epsilon$ is the tensor product
$$K^{\epsilon}=\bigotimes_{i=1}^p K_i,$$
where $K_i$ is the space of holomorphic functions of~$z_i$ for $\epsilon_i=1$, and the space of anti-holomorphic functions of~$z_i$ for $\epsilon_i=-1$, with representations of the one-row Heisenberg group defined on them.

Observe two properties of the Bargmann space~$K^\epsilon$.

(i) The following finite monomials form an orthonormal basis in the Bargmann space:
$$f(z)=\prod_{ij}\frac{x^k_{ij}}{\sqrt{k_{ij}!}},$$
where $x_{ij}=z_{ij}$ for $\epsilon_i=1$ and $x_{ij}=\bar z_{ij}$ for $\epsilon_i=-1$.

(ii) $\int\limits_z f(z)\,d\mu(z)=f(0)$ for every function $f\in K^\epsilon$.

\begin{theorem} The representation $T^\epsilon$ of the group~$N$ in the Bargmann space~$K^\epsilon$ is irreducible.
\end{theorem}

\begin{proof} It suffices to consider only the case $\epsilon\!=\!(1,{\ldots},1)$. We introduce the infinitesimal creation and annihilation operators associated with the representation~$T$, i.e., the operators of the form
$$
A_{ij}^+f=\frac{\partial (Tf)}{\partial z_{ij}^{0}} \bigg|_{z^0=0}, \quad A_{ij}^-f=\frac{\partial (Tf)}{\partial \bar z_{ij}} \bigg|_{z^0=0}.
$$
It follows from the definition that the operators~$A^+$ increase the degree of every monomial by one, while the operators~$A^-$ decrease this degree by one and, therefore, annihilate the vacuum vector. These properties imply that the representation is irreducible.
\end{proof}

\begin{theorem} The representation $T^\epsilon$ of the group~$N$ in the Hilbert space $L^2(Z,\mu)$ is a countable multiple of its restriction to the Bargmann subspace~$K^\epsilon$, i.e., can be decomposed into a countable direct sum of irreducible representations equivalent to the representation of~$N$ in~$K^\epsilon$.
\end{theorem}

\begin{proof}
It suffices to prove the desired assertion for the case of the representation in the space $L^2(Z,\mu)$, where $Z$ is the group of one-row matrices $z=(z_1,\ldots,z_m)$.

With every multiplicator $k=(k_1,\ldots,k_m)$, $k\geq 0$, we associate the Hilbert space~$L_K$ obtained as the norm completion of the set of all linear combinations of monomials of the form $z^{k'}z'^{k''}$ where $k'$ is arbitrary and $k''\leq k$. In particular, $L_0=K$.  It follows from the definition that
\begin{enumerate}
\item[(a)] the subspaces $L_k$ are invariant;
\item[(b)] $L_k \subset L_{k'}$ for $k<k'$; in particular, $ L_k \subset L_{k+e_i},$
$i=1, \dots ,n-1,$ where $\{e_i\}$ is the standard basis in $\mathbb Z^{n-1}$;
\item[(c)] the completion of the inductive limit of~$L_k$ with respect to the embeddings $L_k \subset L_{k'}$ coincides with the space $H^+_1.$
\end{enumerate}
For every multi-index~$k$ and index $i=1, \dots ,n-1$, denote by~$L_{k,i}$ the direct complement to~$L_k$ in the space~$L_{k+e_i}$; that is,
$$
L_{k+e_i}=L_{k,i} \otimes L_k.
$$
Obviously, the spaces $L_{k,i}$ are invariant, and the whole space~$H^+_1$ can be written (not uniquely) as a countable direct sum of such spaces.

Let us check that the representation~$T$ of the group~$N$ in each space~$L_{k,i}$ is equivalent to its representation~$T^+_{1,0}$ in the space~$K$.

Denote $f=\bar z^{k+e_i}.$  This vector lies in~$L_{k+e_i}$
and is orthogonal to the space~$L_k.$ Thus $f\in L_{k,i}.$ One can also easily see that $f$ is a cyclic vector in~$L_{k,i}$.
Analogously, the vector $f_0\equiv 1$ lies in the space~$K$ and is cyclic in this space. Hence it suffices to check that for every element
 $g\in N$, the scalar products
${\langle T(g) f,f \rangle}$ and
${\langle T(g) f_0,f_0 \rangle}$ differ by a factor depending only on $k$. This follows immediately  from the description of the representation operators; namely, according to this description,
\begin{equation*}
{\langle T(g) f_0,f_0 \rangle}=e^{\zeta _0} , \quad
{\langle T(g) f,f \rangle}=e^{\zeta _0}{\langle f,f \rangle}.
\qedhere
\end{equation*}
\end{proof}

\subsection{Description of all nondegenerate irreducible unitary representations of the group~$N$}

Now we will describe the unitary representation~$T_s^\epsilon$ associated with an arbitrary nondegenerate character
 $\chi_S^\epsilon(n)\exp(\tr(\epsilon sns^*))$.

Denote by $\pi_s$  the isomorphism of the space $H$ given by the formula
$$(\pi_s f)(z)=f(sz).$$
We define the operators $T_s^\epsilon(y)$   in the space~$H$ for elements of the group~$N$ by the following equation:
$$T_s^\epsilon(g) f(z) = \pi_s (T^\epsilon(g^s)f(z)).$$
The explicit formulas for the operators $T^\epsilon$ imply the following.

\begin{proposition} The unitary representation of the group~$N$ in the space~$H$ associated with a nondegenerate character~$\chi_s^\epsilon$ is given by the following formula:
\begin{multline}
T_s^\epsilon(n^0,z^0)f(z)
=\exp\bigg(\sum\limits_{i=1}^p\big( (\epsilon sn^0s^*)_{ii} -\frac12 w_i^0 w_i^{0*}\big)\\ - \sum\limits_{i,\epsilon_i=1}w_iw_i^{0*}
-\sum\limits_{i,\epsilon_i=-1}w_i^0w_i^*\bigg) f(w+w^0),
\end{multline}
where we have used the notation $w=sz$, $w^0=sz^0$.
\end{proposition}

Denote by $K_s^\epsilon$ the subspace of all functions $f(w)=f(sz)$ that are holomorphic with respect to the $i$th row of the matrix~$w$ for $\epsilon_i=1$ and anti-holomorphic with respect to this row for $\epsilon_i=-1$.

The formulas for the operators of the representation~$T_s^\epsilon$ imply the following theorem.

\begin{theorem} The spaces $K_s^\epsilon$ are closed invariant subspaces of the Hilbert space~$H^\epsilon$.
\end{theorem}

These subspaces will be called the Bargmann subspaces associated with the characters~$\chi_s^\epsilon$.

Thus, with each nondegenerate character~$\chi_s^\epsilon$ we have associated a unitary representation~$T_s^\epsilon$ of the group~$N$ in the Hilbert space~$K_s^\epsilon$ conjugate
to the representation~$T^\epsilon$ in the space~$K^\epsilon$.

By analogy with the representations~$T^\epsilon$ of the group~$N$ considered above, one can establish that
\begin{enumerate}
\item[(a)]
the representations $T_s^\epsilon$ are irreducible;
\item[(b)]
the space $K^\epsilon_s$ can be written as the tensor product of the Hilbert spaces of functions on~$W_i$, where $W_i$ is the $ i$th row of the matrix~$W$, that are holomorphic for $\epsilon_i=1$ and anti-holomoprhic for $\epsilon_i=-1$;
\item[(c)]
the representation of the group~$N$ in the Hilbert space~$H_s^\epsilon$ is a countable multiple of its irreducible representation in the space~$K_s^\epsilon$.
\end{enumerate}

 \begin{theorem}The representations $T_s^\epsilon$ of the group~$N$ are pairwise nonequivalent.
\end{theorem}

\begin{proof}
Let us compute the spherical functions $\phi_s^\epsilon(n,z)$  of these representations, i.e., the functions
$\phi^\epsilon_s (g)=\langle T^\epsilon_s (g) {\bf 1}, {\bf 1}\rangle,$ where ${\bf 1}$ is the vacuum vector in the space of $T^\epsilon_s$. It follows from the explicit formula for the operators of this representation that $\phi^\epsilon_s (n,z)=\chi_s^\epsilon(n-\frac12 zz^*)$. Since these functions are pairwise distinct, the corresponding representations are pairwise nonequivalent.
\end{proof}

\begin{corollary} The representations $T_s^\epsilon$ of the group~$N$ form a complete system of nondegenerate irreducible unitary representations.
\end{corollary}

To conclude this section, we emphasize that the outlined scheme of considering representations of a general Heisenberg group, i.e., the central extension of a commutative group by a commutative group and a symplectic 2-form, is quite general and does not differ essentially from the representation theory of the classical three-dimensional Heisenberg group. Being nilpotent, such groups have no faithful special irreducible representations: all nontrivial cocycles with values in irreducible representations are just additive characters with values in the one-dimensional identity representation. This result was apparently proved by several authors (see, e.g.,~\cite{Gui}). At the same time, if we give up irreducibility, then nontrivial cohomology appears, for instance, in the regular representation or in any unitary representation weakly containing the identity representation. We will use this fact when studying representations of the Iwasawa group, i.e., the semidirect product of a nilpotent group and the Heisenberg group.

\section{Nonunitary representations of the Iwasawa group~$P$}

\subsection{Almost invariant measures on the group~$S$
and nonunitary representations}

\begin{definition} We say that a measure $d\nu(s)$ on~$S$ is almost invariant if it is quasi-invariant with respect to the transformations
 $s\rightarrow s_0$ of the group~$S$ and its derivatives $\frac{d\nu(ss_0)}{d\nu(s)}$ are bounded for every $s_0\in S$.
\end{definition}

In particular, the Haar measure on~$S$, invariant with respect to the right translations $s\rightarrow ss_0$, is an almost invariant measure on~$S$.

With each $S$-orbit in the space of nondegenerate unitary representations~$T^\epsilon_s$ and each almost invariant measure~$d\nu (s)$ on~$S$ we associate a nonunitary representation of the group~$N$.

By definition, this representation is realized in the direct integral
\begin{equation}
{\mathcal K}^\epsilon_s=\int\limits_S^\bigoplus K^\epsilon_s\, d\nu(s)
\end{equation}
over $S$ of the Hilbert spaces~$K^\epsilon_s$ acted on by the representations~$T^\epsilon_s$ of~$N$.

The actions on $K^\epsilon_s$ of the unitary representations~$T^\epsilon_s$ induce a unitary representation of the group~$N$ on the whole space~$\mathcal K^\epsilon$.

Further, we define the operators $T^\epsilon_s$ corresponding to elements of the group~$S$ in the space~$\mathcal K^\epsilon$ as right translation operators:
$$T^\epsilon(s_0)f(s)=f(ss_0) \quad \text{for every} \quad s_0\in S.$$
Since the measure $d\nu (s)$ is almost invariant, it follows that these operators are defined and bounded on the whole Hilbert space~$\mathcal K^\epsilon$.

One can easily check that the operators thus defined for the subgroups~$N$ and~$S$ together generate a (nonunitary in general) representation $T=T_\epsilon$ of the whole Iwasawa group $P=S\rightthreetimes N$. Clearly, this representation is unitary in the case where
 $d\nu(s)$ is the Haar measure.

 \begin{theorem} The representations $T_\epsilon$
 of the Iwasawa group~$P$ are operator-irreducible and pairwise nonequivalent.
\end{theorem}

\begin{proof}
Since the representations~$T_s^\epsilon$ of the group~$N$ are pairwise nonequivalent for different~$\epsilon$ and~$s$, it follows that
the representations $T_\epsilon$ are pairwise nonequivalent.

Let us prove that they are irreducible. Let~$A$ be an arbitrary bounded operator in the space~$H^\epsilon$ that commutes with the operators of~$T_\epsilon$. Since the representations~$T_s^\epsilon$ of the subgroup~$N$ are pairwise nonequivalent, it follows that this operator is a multiple of the identity operator on each subspace~$K_s^\epsilon$.

Since the measure $\nu(s)$ on~$S$ is ergodic, this operator is constant.
\end{proof}

\begin{theorem}
The representation $T$ of the group~$P$ associated with an almost invariant measure~$\nu$ is equivalent to the representation of this group in the Hilbert space $L^2(S,\mu, K)$ associated with the Haar measure, where the representation operators corresponding to elements of the subgroup~$N$ are unitary and have the same form, and the operators corresponding to elements of the subgroup~$S$ are given by the formula
$$
(T(s_0)F)(s)=F(ss_0)\frac{a(s)}{a(ss_0)}, \quad \mbox{where} \quad |a(s)|^2=\frac{d\mu(s)}{d\nu(s)}.
$$
\end{theorem}

\begin{proof}
We use the following equation:
$$
\int\limits_S|f(ss_0)|^2\,d\nu(s)=\int\limits_S|f(ss_0)a(s)|^2\,d\mu(s).
$$
Putting $F(s)=f(s)a(s)$, we can write the right-hand side in the form
$$
\int\limits_S|F(ss_0)\frac{a(s)}{a(ss_0)}|^2\,d\mu(s).
$$
The theorem follows immediately.
\end{proof}

\section{Special nonunitary representations of the Iwasawa group}

\subsection{Definition of a special representation}
\begin{definition}
A  (not necessarily unitary) representation $T$ of a locally compact group~$G$ in a Hilbert space~$H$ is called special if its $1$-cohomology group is nontrivial; this means that there exists a nontrivial continuous map $\beta:G\rightarrow H$ satisfying the relation
$$\beta(g_1g_2)=T(g_1)\beta(g_2)+\beta(g_1)$$
(a $1$-cocycle). The nontriviality means that this cocycle cannot be written in the form $\beta(g)=T(g)\xi-\xi$ with $\xi\in H$.
\end{definition}

The property of being special for a representation~$T$ of a group~$G$ is crucial for constructing a representation of the group~$G^X$ of $G$-currents, i.e., the group of maps $G\rightarrow X$, where $X$ is a space with a probability measure, with the pointwise multiplication. Namely, if $T$ is an irreducible unitary special representation of~$G$, then a Fock construction associates with~$T$ an irreducible representation of~$G^X$.

The purpose of this paper is to select, in the class of almost invariant measures~$\nu(s)$ on the group~$S$, those measures for which the representation of the group~$P$ in the space~$H^\epsilon$ described above is special, and then extend it to a representation of the corresponding group~$U(p,q)$.

\subsection{Sufficient conditions for a representation of the Iwasawa group~$P$ associated with an almost invariant measure to be special}

For simplicity, we restrict ourselves to representations in the spaces~$\mathcal K^\epsilon$ with $\epsilon=(1,1,\ldots,1)$.  In this case,
$$
{\mathcal K}=\int\limits^\otimes_S K_s \,d \nu(s),
$$
where $K_s = K$ are the spaces of holomorphic functions of~$Z$.

The irreducibility
criterion for the case of an arbitrary~$\epsilon$ is analogous.

 \begin{theorem} Let $\nu$ be an almost invariant measure on~$S$ for which there exists a scalar positive function~$f(s)$ satisfying the following conditions:
\begin{enumerate}
\item[\rm(i)]
$\int\limits_Sf(s)^2\,d\nu(s)=\infty$;
\item[\rm(ii)]
$\int\limits_S |f(ss_0)-f(s)|^2\,d\nu(s)<\infty$ for every $s_0\in S$;
\item[\rm(iii)]
$\int\limits_S (1\!-\!\Re e^{\tr(s(n-\frac12 zz^*)s^*)} )f^2(s) d\nu(s)\!<\!\infty$
for every  $(n_0,z_0)\!\in\!N$.
\end{enumerate}

Then the representation~$T$ of the Iwasawa group associated with the measure~$d\nu (s)$ has a nontrivial $1$-cocycle of the form
$$\beta(g)=T(g) f(s) {\bf 1} -f(s){\bf 1}$$
and thus is a special representation.
\end{theorem}

\begin{proof}
Condition~(ii) is obviously equivalent to the requirement that the value~$\beta(s_0)$ belongs to the space of~$T$ for every
$s_0\in S$. Let us check that condition~(iii) is equivalent to the requirement that the value~$\beta(g)$ belongs to the space of~$T$ for  $g\in N$. Indeed, the latter condition has the form
\begin{equation}
\int\limits_S \|T_s(g){\bf 1} - {\bf 1}\|^2_K |f(s)|^2 \,d\nu(s) < \infty \quad\text {for}\quad g\in N.
\end{equation}
Since $\int\limits_S \|T(n_0,z_0){\bf 1}-{\bf 1}\|_{K_s}^2=2(1-\Re  e^{\tr(sns^*-\frac12 zz^*)}),$ this condition is equivalent to~(iii). Thus we have proved that $\beta$ is a 1-cocycle. From~(i) it follows that it is nontrivial.

As a by-product, we have obtained an explicit formula for the norm of this nontrivial cocycle:
\begin{equation}
\|\beta(s_0)\|^2=\int\limits_S |f(ss^0)-f(s)|^2 \,d\nu(s) \quad\text{for}\quad s_0\in N,
\end{equation}
\begin{align}
\|\beta(n^0,z^0)\|^2 = \int(1-\Re\tr(s(n-\frac12 zz^*)s^*f^2(s)) \,d\nu(s)&
\\ \quad\text{for}\quad (n^0,z^0)\in N.&\notag
\qedhere
\end{align}
\end{proof}

\subsection{An example of a special representation}

Let us give an example of a special representation. Consider the almost invariant measure on~$S$ given by the formula
$$d\nu(s)=|s|^{-p^2}ds, \quad\text{where}\quad |s|^2=\tr(ss^*)=\sum\limits_{i\geq j} |s_{ij}|^2$$
and $ds$ is the Lebesgue measure on~$S$.

Let us prove that the nonunitary representation of the subgroup~$P$ associated with this measure has the nontrivial 1-cocycle
$$\beta(p)=T(p)f-f, \quad\text{where}\quad f=e^{-\frac12|s|^2}\bf1.$$

For this, it suffices to check that the function~$f(s)$ satisfies the above conditions~(i)--(iii).

Converting from rectangular coordinates on~$S$ to polar coordinates
in the formula for the measure $d\mu(s)$ yields
$s=r\omega$, where $r=|s|$ and $|\omega|=1.$
In these coordinates, conditions~(i)--(iii) on~$f$ take the following form:
\begin{enumerate}
\item[\rm(i)]
$\int\limits_0^\infty e^{-r^2}r^{-1} \,dr=\infty$;
\item[\rm(ii)]
$\int\limits_S |e^{-\frac12 r^2|\omega s_0|^2}  - e^{-\frac12 r^2}|^2 \,r^{-1}\,drd\omega<\infty$;
\item[\rm(iii)] $\int\limits_S (1-e^{-\frac{r^2}{2}\tr(sz^0z^{0*}s^*)})r^{-1} \,dr < \infty$.
\end{enumerate}
These three conditions are obviously satisfied.

Integrating~(ii) and~(iii) with respect to~$r$, we obtain the following expressions for the norm of the nontrivial 1-cocycle:
\begin{enumerate}
\item
$\| \beta(s_0)\|^2=\int\limits_\Omega \log \frac{|\omega s^0|^4}{|\omega s^0|^2+1} \,d\omega$;
\item
$\| \beta (n^0,z^0)\|^2 \int\limits_\Omega \log\frac{|n^0-\frac12 z^0 z^{0*}|^4}{|z^0z^{0*}|^2}\,d\omega$.
\end{enumerate}

\section{Extension of a special representation of the Iwasawa subgroup to the whole group~$U(p,q)$}

\subsection{Setting of the problem}

We consider a
  special nonunitary representation~$T$, introduced in the previous section, of the Iwasawa group~$P$ in a Hilbert space~$K$ with a nontrivial 1-cocycle and the subspace~$L$ generated by the values of this cocycle. The problem we solve is to extend the representation~$T$ of the group~$P$ in the space~$L$ to the whole group~$U(p,q)$.

The construction  is based on the following property of the Iwasawa group. As we have already observed, every element~$g\in U(p,q)$ can be uniquely written as the product $g=kp$ of elements~$k$ and~$p$, where $p\in P$ and~$k$ belongs to the maximal compact subgroup~$K$ in~$U(p,q)$, determined by the additional relation $kk^*=e_p$.

The operators of this extension are defined not on the whole Hilbert space $L^2(S,\mu, K)$, but only on a pre-Hilbert space~$L$ linearly spanned by the vectors~$b(p,r)$ determined by the nontrivial 1-cocycle. Recall that, by Lemma~1,
 these vectors are pairwise distinct and linearly independent, i.e., form a (nonorthogonal in general) basis in the linear space~$L$.

The operators~$T(k)$, $k\in K$, will be defined on the set of vectors~$b(p)$ as permutations.

We consider the special nonunitary representation~$T$ of the group~$P$ in the space  $L^2(S,\mu, K)$ described above with the nontrivial 1-cocycle
$$b:p\rightarrow L^2(S,\mu, K),$$
$$b(p)=T(g)f-f, \quad\text {where}\quad f=e^{-\tr(ss^*)}.$$

Since $T$ is irreducible, the vectors~$b(p)$ form a total subset in $L^2(S,\mu, K)$. The desired extension to the group~$U(p,q)$ will be defined only on the subspace~$L$ linearly spanned by these vectors.

 \begin{theorem}
The map $p\rightarrow b(p)$ is injective. The image of every element $p=(s,n,z)$ of the group~$P$ under this map has the following form:
$$
b(s,n,z)=\exp(-\tr(sas^*+h)), \ \text {where }\ a=s(n-\frac12 z z^*+1)s^*, \enskip h=z^*s^*.
$$
\end{theorem}
\begin{proof} By the definition of~$T$, we have
$$b(s,n,z)=\
\exp(-\tr(ss_0(n_0-\frac12z_0z_0^*)s_0^*-zz_0^*s_0^*-ss_0s_0^*s^*))-\exp(-\tr(ss^*)).
$$
The obtained formula is equivalent to that from the statement of the theorem. Let us check that the equality $b(p')=b(p)$ is possible only if $p'=p$. Indeed, if $p'=s_0'n_0'z_0'$, then, by the formula for the 1-cocycle, we have
$$
s_0'(n_0'-\frac12 z_0'z_0'^*-1)z_0'^*=s_0(n_0-\frac 12 z_0z_0^*-1)z_0^*,\quad s_0'z_0'=s_0z_0.
$$

Recall that $n_0$ and $n_0'$ are skew-Hermitian matrices; hence the first equality implies that
 $n_0's_0'^*=s_0n_0s_0^*$ and $(z_0's_0')(s_0'z_0')^*=(z_0s_0)(s_0z_0)^*$.
It follows immediately  that $s_0'=s_0$, and hence $z_0'=z_0$ and $n_0'=n_0$, i.e., $p'=p$.
\end{proof}

 \begin{theorem}
Let $T$ be a
nonunitary representation of the group~$P$ in the Hilbert space $L^2(S,\mu,K)$, and let $b(p)$ be a $1$-cocycle. Then this representation of~$P$ and the $1$-cocycle~$b(p)$ can be extended to the whole group~$U(p,q)$.
\end{theorem}

To describe this extension, we use the following lemma.

\begin{lemma} The vectors $b(p)$ are pairwise distinct and linearly independent for $p\neq 0$.
\end{lemma}
\begin{proof}
The explicit expression for the representation operators implies that these vectors can be written in the canonical form of the cocycle~ $b(p)$ for $p=(s_0,n_0,z_0)$.
One can easily see that $b(p)\neq b(p')$ for $p\neq p'$.
\end{proof}

\subsection{Description of the representation}

We will describe the extension of the special representation of the Iwasawa group~$P$ in the space $L^2(S,\nu, K)$ and its nontrivial 1-cocycle $\pi(p)=T(p)f-f$, where $f=e^{-\frac12|s|^2}$, to  the whole group~$U(p,q)$.

The desired representation is realized not in the whole Hilbert space $L^2(S,\nu, K)$, but only on the invariant subspace~$L$ linearly spanned by the vectors~$\beta(p)$.

\begin{definition}
We define the action of the representation operators for elements of the subgroup~$K$ on the set of vectors~$b(p)$ by the formula $T(k)b(p)=b(p')$, $p'\in P$, where $p'$ is defined by the relation $kp=p'k^*$, and extend these operators by linearity to the whole space~$L$.
\end{definition}

This action is well defined, as follows from the following property of the group~$P$.

\begin{lemma}
The map $p\rightarrow b(p)$ is bijective, and the vectors~$b(p)$ are linearly independent.
\end{lemma}

The lemma follows from the following presentation of the vectors $b(p)=b(s,n,z)$:
$$
b(s,n,z)=\exp(-\tr s^*sa-h),\quad\text{where}\quad a=s(n-\frac12z^*-1), \quad h=sz.
$$
It is not difficult to check that the constructed operators for the elements of the subgroup~$K$ define, together with the operators for the elements of the subgroup~$P$, a representation of the whole group~$U(p,q)$.

An extension of the 1-cocycle~$b$ to the whole group~$U(p,q)$ is given by the formula $b(p,k)=b(p)$ for any $p\in P$, $k\in K$.

\subsection{The representation operator corresponding to the involution~$w$}

It is well known that the group~$U(p,q)$ is algebraically generated by the Iwasawa subgroup~$P$ and a single element of the compact group~$K$, the involution~$w$ (cf.~(1)); hence, in order to extend the representation of~$P$ in the space~$L$ to~$U(p,q)$, it suffices to describe only the operator~$T(w)$. Let us describe the action of this operator.

It follows from the definition that the action of~$T(w)$ on the vectors $b(p)\in P$ is given by the formula
$$
T(w)b(p)=b(p'),
$$
where $p'\in P$ is determined by the relation $kp=p'k^*$.
In other words, the involution commutes with the cocycle.

It is more convenient to use another interpretation of the space~$L$ and the action of the operator~$T(w)$ on~$L$. Namely, let us introduce the space~$H$ of Hermitian matrices of the form $pp^*$, $p\in P$. Since the map $p\rightarrow |n|=pp^*$ is a bijection, the group structure on~$P$ induces a group structure in the space~$H$:
$$
n_1  n_2=pn_2p^*,
$$
where $p\in P$ is determined by the relation $pp^*=n_1.$

Note that the map $n\rightarrow b(n)$ preserves the space~$H$. As a result, the space~$L$ can be interpreted as the space linearly spanned by the independent vectors $b(n)$, $n\in H$. The action of the operators corresponding to elements of the subgroup~$P$ is given by the formula
$$
T(p)b(n)=b(pnp^*),
$$
and the operator corresponding to the involution takes the form
$$
T(w)b(n)=b(p^*np).
$$

\subsection{Construction of an extension of a representation of a subgroup in a free $\mathbb C$-module}

Digressing a little from our main course, we will show that the suggested construction of an extension of a representation of a subgroup to an ambient group is a special case of a construction applicable to free $\mathbb C$-modules. Every group~$P$ gives rise in a natural way to the free $\mathbb C$-module $\mathbb C[P]$ over the set of elements $p\in P$. Its elements are finite or countable sums of the form $\sum \lambda_i p_i$, where $\lambda_i\in \mathbb C$ and $p_i\in P$. This module is acted on by the representation~$T$ of the group~$P$ by the operators defined on the basis elements by the formula
$$T(p)p=p_0p.$$

Each element $p_0$ gives rise to the nontrivial cocycle of~$P$ of the form
$$
\beta(p)=T(p)p_0-p_0
$$
and the linear space~$L$ of codimension~1 in  $\mathbb C[P]$ linearly spanned by the vectors~$\beta(p)$.

Now let $G$ be an arbitrary topological group that can be written as the product (in the topological sense) of subgroups~$P$ and~$K$ whose intersection contains only the identity elements. Then we can consider the representation of~$K$ in $\mathbb C[P]$ defined on the basis vectors~$p$ by the following formula: $T(k)p=p'$, where $p'\in P$ is determined by the relation $kp=p'k'$, $k'\in K$. The operators thus defined can be extended by linearity to the whole module, and they generate, together with the representation operators corresponding to the elements of the subgroup~$P$, a representation of the whole group~$G$.

 \begin{theorem}If the action of the subgroup~$K$ on the group~$P$ is ergodic, then the constructed representation of the group~$G$ is operator-irreducible.
\end{theorem}

\smallskip\noindent
{\bf Remark.} The space $L$ in this definition coincides with the space~$L$ from the example considered above. The subgroup~$K$ is the subgroup of elements satisfying the additional condition $kk^*=1$.
\smallskip

Let $T$ be an arbitrary faithful representation of the subgroup~$P$ in a space~$H$ and $\beta$ be an arbitrary nonzero 1-cocycle of~$T$. Denote by $B\subset H$ the set of values of this 1-cocycle, and by~ $H_0$, the subspace in~$H$ linearly generated by the elements of~$B$.

\begin{theorem} The representation~$T$ of the subgroup~$P$ in the space~$H_0$ and the $1$-cocycle~$\beta$ of~$P$ in~$H_0$ can be extended to the whole group~$U(p,q)$.
\end{theorem}

Let us describe a construction of this extension. Since $T$ is a  faithful representation of~$P$, it follows that the elements $\beta(p) \in B$ are pairwise distinct. We define the operators $T(k)$, $k\in K$, on the set~$B$ as permutations of this set:
$$T(k)\beta(p)=\beta(p'), \quad p'\in P,$$
where $p'$ is determined by the relation $kp=p'k'$, $k\in K$.

Since the vectors $\beta(p)$ are pairwise distinct, the operators~$T(k)$ are well defined.

It follows from the definition that these operators have the group property on~$B$, and hence they can be extended to a representation of~$K$ on the whole space~$H_0$.

Together with the operators corresponding to the elements of the subgroup~$P$, they generate the desired extension of the representation~$T$ of~$P$ to the whole group~$U(p,q)$.

An extension of the 1-cocycle~$\beta$ to~$U(p,q)$ is given by the following formula: $\beta(pk)=\beta(p)$ for any $p\in P$ and $k\in K$. In particular, $\beta(k)=0$ for every $k\in K.$

\subsection{Properties of the constructed extensions}

\begin{enumerate}
\item
The representation operators of the group~$U(p,q)$ in the subspace~$H_0$ are, in general, unbounded and cannot be extended to a representation of~$U(p,q)$ on the whole space~$H$. However, if the operators corresponding to the elements of the subgroup~$K$ are unitary on the subset~$B$ and the operators of the original representation of the subgroup~$P$ are also unitary, then such an extension  exists and determines a unitary representation of the group~$U(p,q)$.

\smallskip
\noindent{\bf Remark.} Since simple Lie groups of rank greater than one have no special unitary representations, such a case is possible only for the group~$U(1,q)$.
\smallskip

\item
It follows from above that it suffices to know only a formula for the operator~$T(w)$:
$$
T(w)\beta(p)=\beta(p'),
$$
where $p'\!\in\! P$ is determined by the equality $w p\!=\!p'k$, ${k\!\in\! K}$. This equality implies the relation
 $w pp^* w =(p')^*p'$, which uniquely determines~$p'$ from~$p$.

\item
When describing the extension, we do not impose any conditions on the original 1-cocycle~$\beta$. However, if the original representation is special, then its extension to the group~$U(p,q)$ may lose this property of being special.

It is known that for $p>1$, the group $U(p,q)$ has no special unitary representations. Hence if the original representation of~$P$ is unitary and nondegenerate, then the constructed extension to~$U(p,q)$ is not unitary.

\item
A unitary faithful representation of the subgroup~$P$ can be nonspecial. However, this group has degenerate special representations. Indeed, it is known that a unitary representation of the group~$P_2$ of lower triangular unimodular matrices of the second order is special. Meanwhile, this group is a quotient of the solvable group~$P$, and hence the special representation of~$P_2$ can be extended to a degenerate representation of~$P$.
\end{enumerate}

We have completed the first part of our construction, by having constructed a nonunitary special representation of the group~$U(p,q)$
in a pre-Hilbert space. Below we proceed to the construction of the corresponding representation of the group of currents.

\section{Representations of the group of currents~$P^X$}

\subsection{General remarks on representations of groups of currents}

The group of currents $G^X$, where $G$ is an arbitrary locally compact group and $X$ is a standard space with a probability measure~$m$, is the group of measurable maps $X\rightarrow G$ endowed with the pointwise multiplication. There exists a connection, discovered long ago, between nontrivial 1-co\-ho\-mo\-lo\-gy of the group~$G$ in irreducible unitary representations and irreducible representations of the corresponding groups of currents~$G^X$
(see \cite{Ar, VGG1, VGG2, Ism} etc.). Namely, every special irreducible unitary representation of~$G$ gives rise to an irreducible unitary representation of~ $G^X$ in the Fock space.

It is known that the 1-cohomology of any simple Lie group~$G$ of rank greater than one in all irreducible representations is trivial, and hence there are no irreducible unitary representations of the groups of $G$-currents in the Fock space. In this section, taking as an example the group~$U(p,q)$, we will construct, at the expense of giving up unitarity, a nonunitary representation of the group of currents~ $U(p,q)^X$.

We begin with a construction of a nonunitary representation of the group of currents~$P^X$, and then extend it to a nonunitary representation of the group $U(p,q)^X$. The construction starts from an arbitrary nonunitary representation of~$P$. We repeat that the question of whether there exists a special unitary irreducible representation of this group is still open.

Since $P$ is a semidirect product of the groups~$S$ and~$N$, it is most convenient to use the quasi-Poisson model, which was introduced in~\cite{VG7} and applied there for the case of Lie groups of rank~1. Then this construction can be easily carried over to the classical Fock model of representations of groups of currents. We begin with two general definitions.

\subsubsection{The quasi-Poisson model of the Fock space and the countable product of punctured Hilbert spaces}

By definition, a quasi-Poisson measure is an infinite $\sigma$-finite measure~$\sigma$ defined by a triple $(Y,\mu,u)$, where $Y$ is a standard Borel space, $\mu$ is a measure on~$Y$, and $u$ is a positive function on~$Y$ such that
$$
\int\limits_Y e^{-u(y)}\,d\mu(y)=\infty.
$$
This is the measure on the space~$\mathcal E(Y)$  of countable or finite sequences in~$Y$ (the space of configurations) given by the following characteristic functional:
\begin{equation}
\int\limits_{\mathcal E(Y)}\exp \big(-\sum\limits_{y\in \omega} f(y)\big) \,d\sigma(\omega)=\exp\bigg(\int\limits_Y(e^{-f(y)}-e^{-u(y)})d\mu(y)\bigg).
\end{equation}

Note that for $u\equiv 0$, this definition coincides with that of the classical Poisson measure associated with the pair~$(Y,\mu)$. The quasi-Poisson space associated with this measure is, by definition, the Hilbert space
$L^2(\mathcal E(Y),\sigma)$.

The countable tensor product of Hilbert spaces~$H_i$ with fixed unit vectors~$h_i\in H_i$  is the completion of the inductive limit of the finite tensor products $\otimes_{i=1}^n H_i$ with respect to the embeddings $\otimes_{i=1}^n H_i \rightarrow \otimes_{i=1}^{n+1} H_i$.

If a sequence of vectors $a_1\otimes\ldots\otimes a_n$ converges in the space thus defined, then its limit will be written as the infinite product $\otimes_{i=1}^\infty a_i$. The vectors obtained in this way form a total subset in the infinite tensor product.

This part of the construction should be regarded as standard in the theory of infinite tensor products.

\subsection{The quasi-Poisson space and nonunitary representations of the group~$P^X$}

Let $U$ be the special nonunitary representation of the Iwasawa group~$P$ associated with a pair $(T,\mu)$, where $T$ is the Bargmann representation of the group~$N$ in the space~$K$ and $\mu$ is an almost invariant measure on~$S$. Let $b(g)$ be the nontrivial 1-cocycle of this representation given, according to Theorem~10, by the formula
$$b(g)=T(g)f-f, \quad\mbox {where}\quad f(s)=e^{-(us)}1.$$

Let us describe the quasi-Poisson space in which we will realize the representation of the group of currents~$P^X$ associated with this special representation of~$P$. Consider the triple $(Y,\mu,u)$ where
$Y=S$ while~$\mu$ and~$u$ are, respectively, the almost invariant measure on~$S$ and the positive function on~$S$ from the definition of the original representation of the Iwasawa group~$P$.

Denote by $\sigma$ the quasi-Poisson measure on $\mathcal E(Y)$ associated with this triple.

Further,  with each configuration $\omega\in \mathcal E(Y)$ we associate the countable tensor product
$$
K_\omega^{\otimes}=\otimes_{s,x\in\omega}K_s,
$$
where $K_s=K$ is the Bargmann space.

The quasi-Poisson space associated with the measure~$\sigma$ is defined by the formula
$$
\operatorname{QPS}(\mathcal E(Y),\sigma,K^{\otimes})=L^2(\mathcal E(Y), \sigma, K^{\otimes}),
$$
i.e., as the space of sections of the fiber bundle over~$\mathcal E(Y)$ with fibers~$K_\omega^{\otimes}$ endowed with the norm
\begin{equation}
\|(f)\|^2=\int\limits_{\mathcal E(Y)}\|f(\omega)\|^2_{K_\omega^{\otimes}}\,d\sigma(\omega).
\end{equation}

\subsection{Formulas for the representation operators of the group~$P^X$}
We will define a representation~$\widetilde U$ of the group~$P^X$ associated with the representation~$U$.
Since $P^X=S^X \rightthreetimes N^X$, it suffices to describe it for the groups~$N^X$ and~$S^X$ separately.

We begin with the case of~$N^X$.

The action of the operators~$T(n)$ for $n\in N$ in the Bargmann space~$K$ gives rise in a natural way to an action of the elements of the group~$N^X$ on each component $K_w^\otimes$, $\omega\in \mathcal E(Y)$; namely, on the total subset of vectors of the form
 $\otimes_{(s,x)\in \omega}f_{s,x}$, the corresponding operators  act as multiplicators.

Clearly, these operators are unitary on each tensor product~$K_\omega^\otimes$ and generate unitary representations of the group~$N^X$ on the whole Hilbert space $\operatorname{QPS}(\mathcal E(Y),\sigma,K^\otimes)$. Thus the representation~ $\widetilde U$ is defined.

Let us impose the following additional condition on~$S^X$, $\mu$, and $u$:
\begin{equation}
\exp\bigg(\int\limits_{SX}(e^{-u(ss_0)(x)}-e^{-u(s)})\,d\mu(s)d\mu(x)\bigg)<\infty.
\end{equation}

We define the operators  in the Hilbert space $\operatorname{QPS}(S(Y),\sigma,K)$ corresponding to elements of the subgroup~$S^X$ as translations:
\begin{equation}
\widetilde U(s_0(\,\cdot\,))f(\omega)=f(\omega s_0)(\,\cdot\,).
\end{equation}

\begin{theorem} For every element $\widetilde  s=s_0(\,\cdot\, )\in S^X$, the quasi-Poisson measure~$\sigma $ on $\mathcal E (S \times X)$ satisfies the bound
$$
 d  \sigma (\omega \widetilde  s) < c(\widetilde  s) \,d \sigma (\omega ),
$$
where $c(\cdot)$ is a bounded function on~$S^*$.
\end{theorem}

 \begin{proof} It follows from the formula for the characteristic functional of the measure~$\sigma $
that
\begin{multline*}
\int\limits_{\mathcal E(S \times X)} \exp\bigg( - \sum\limits_{(s,x)\in \omega }f(s,x)\bigg)
\,d \sigma (\omega\widetilde  s )
\\
=
\exp\bigg(\int\limits_{S \times X } (e^{ -  f(ss_0^{-1})}-e^{-u(s)})\,d\mu(x)
dm(x) \bigg).
\end{multline*}
The right-hand side of this equality can be written as the product
$$
\exp\bigg(\int\limits_{S \times X } (e^{ -  f(ss_0^{-1})}-e^{-u(s)})\,d\mu(x)
dm(x) \bigg)=J_1 J_2,
$$
where
$$
J_1=\exp\bigg(\int\limits_{S \times X } (e^{ -  f(s)}
-e^{-u(s)})\,d\mu(ss_0x) dm(x)\bigg),
$$
$$
J_2=\exp\bigg(\int\limits_{S \times X } (e^{ -  u(ss_0^{-1})}-e^{-u(s)})\,d\mu(x)
dm(x) \bigg).
$$
It follows from condition~(ii) and the formula for the characteristic functional of~$\sigma $
that
$$
J+1 \leq c(\widetilde  s) \int\limits_{{\mathcal E}(S \times X)}
\exp\bigg( - \sum\limits_{(s,x)\in \omega }f(s,x)\bigg)\,d \sigma (\omega ).
$$
Condition~(i) implies that the integral~$J_2$ converges. Thus we obtain the bound
\begin{multline*}
\int\limits_{\mathcal E(S \times X)} \exp\bigg( - \sum\limits_{(s,x)\in \omega }f(s,x)\bigg)
d \sigma (\omega\widetilde  s )
\\
\leq c \int\limits_{\mathcal E(S \times X)}
\exp\bigg( - \sum\limits_{(s,x)\in \omega }f(s,x)\bigg)\,d \sigma (\omega ),
\end{multline*}
where $c$ is a constant.

It follows that  $d \sigma (\omega\widetilde  s ) < c \,d \sigma (\omega )$.

This property of the measure~$\sigma $ will be called the  {\it almost projective equivalence with respect to the transformation
group~$S^X$}. Note that if $\nu$ is an invariant measure, then we have
$d \sigma (\omega\widetilde  s ) =c(\widetilde  s) c \,d \sigma (\omega ),$  where
$c(\widetilde  s)$ is a character on~$S^X.$
\end{proof}

\begin{definition} We define a representation of the group~$S^X$
in the quasi-Poisson space $\operatorname{QPS}(\mathcal E(Y),\sigma,K)$ by the formula
$$
(U(\widetilde  s)f)(\omega )=f(\omega \widetilde  s), \quad  \widetilde  s\in S^X.
$$
\end{definition}

The above theorem implies that the operators~$U(\widetilde  s)$ for $\widetilde  s\in S^X$ are bounded.

It is not difficult to see that these operators, together with the operators corresponding to the elements of the group~$N^X$, generate a representation of the whole group $P^X=S^X\rightthreetimes N^X$ by bounded operators in the quasi-Poisson space
 $\operatorname{QPS}(\mathcal E(Y),\sigma,K)$.

 \subsection{Formulas for the representation operators  of the group~$U(p,q)^X$}
 Let us give a unified formula for the operators of the representation of~$U(p,q)^X$
on some total subset in $\operatorname{QPS}(\mathcal E(Y),\sigma,K)$.

Denote by  $K^Y$, where $Y=S \times X$,  the set of maps
$v:\,S \times X\to K$ such that
\begin{enumerate}
\item[\rm(i)]  $f_v(\omega )\equiv \bigoplus_{(s,x)\in \omega } v(s,x) \in K^\oplus _\omega  $
for almost all configurations $\omega \in \mathcal E(Y)$;
\item[\rm(ii)]  $\int\limits_{\mathcal E{(Y)}} \|f+v(\omega )\|^2 \,d \sigma (\omega  )  < \infty $.
\end{enumerate}

Thus we have a map $K^Y \to \operatorname{QPS}(\mathcal E(Y),\sigma,K)$
of the form
$$
v \mapsto  f_v(\omega )=\bigoplus_{(s,x)\in \omega } v(s,x).
$$
The set of functions $f+v(\omega )$ thus defined is total in the space $\operatorname{QPS}(\mathcal E(Y),\sigma,K)$; hence, in order to describe the representation~$U$ of the group~$U(p,q)^X$, it suffices to describe its action only on the functions from this set.

The original representation~$\widetilde  T$ of the group~$U(p,q)$ in the space~$\mathcal K$ induces a pointwise representation of the group of currents~$U(p,q)^X$ in the space~$\widetilde  K^Y$ of all sections~$v(s,x)$ of the fiber bundle over $S \times X$
with fiber~$K$, whose operators will be denoted by the same symbol~$\widetilde  T,$ i.e.,
$$
(\widetilde  T(g(\,\cdot\, ))v)(s,x)=\widetilde  T(g+x) y(s,x),
$$
where the operator in the right-hand side is the operator of the representation of the group~$U(p,q)^X$
acting on~$v$ as a function of~$s$.

\begin{theorem} The operators $\widetilde  T$ of the group~$U(p,q)^X$
preserve the set~$K^Y$; the action of
the operators~$U(\widetilde  g)$ of the representation of~$U(p,q)^X$ in the quasi-Poisson space $\operatorname{QPS}(\mathcal E(Y),\sigma,K)$ on the subset of functions~$f_v(\omega)$ is given by the following formula:
\begin{equation}
(U(\widetilde  g)f_v)(\omega )=f_{\widetilde  T(\widetilde  g)} (\omega).
\end{equation}
\end{theorem}

Denote by~$A$  the operators of the affine representation of~$U(p,q)^X$
in the space~$F_\nu^X$ associated with the linear representation~$T$ and its 1-cocycle~$b$:
$$
A(g)v= T(g)v +b(g).
$$





\subsection{Irreducibility conditions for the representation~$U$ of the group~$U(p,q)^X$}

\begin{theorem} If a representation~$T$ of the subgroup~$N$ is irreducible and the representations~$T_s$ conjugate with it are pairwise nonequivalent, then the representation~$\widetilde U$ of the group
$P^X= S^X\rightthreetimes N^X$ in the quasi-Poisson space
$\operatorname{QPS}(\mathcal E(Y),\sigma,K)$ is irreducible.
\end{theorem}

\begin{proof}
Since the representation~$T$  of the subgroup~$N$ is irreducible, it follows that for almost every configuration $\omega \in\mathcal E (Y) $, the representation~$U_\omega $ of the group~$N^X$ in the space~$K^\oplus_\omega $ is irreducible.

Since the representations~$T_s$ are pairwise nonequivalent, it follows that the representations~$U_\omega $ of the group~$N^X$ in the spaces~$K^\oplus_\omega $ are pairwise nonequivalent. Hence every bounded operator~$A$ on
$\operatorname{QPS}(\mathcal E(Y),\sigma,K)$ that commutes with the operators of the subgroup~$N^X$ is the operator of multiplication by a function~$f(\omega ).$  If~$A$ commutes also with the operators of the subgroup~$S^X,$ then the function~$f(\omega )$ is constant on the orbits of this subgroup in~$\mathcal E(Y).$ Together with the ergodicity of the quasi-Poisson measure~$\sigma $ with respect to the transformations from the group $S^X$, this implies that
$ f(\omega )=\mathrm{const}.$
\end{proof}

\section{Construction of an extension of a nonunitary representation of the group~$P^X$ to a representation of the group~$U(p,q)^X$}

\subsection{Setting of the problem}

We proceed to the final stage of our construction, that of constructing an extension of a representation from the group~$P^X$ to the whole group of currents~$U(p,q)^X$.

To this end, we replace the original quasi-Poisson space $\operatorname{QPS} (\mathcal E(Y),\sigma,K^\otimes)$ with an invariant pre-Hilbert subspace  linearly spanned by a fixed vector and its group translations. This approach is reasonable, since, by the irreducibility of the representation of the group~$P^X$,
every such subspace is dense in the space of the original representation. We will describe an extension of the representation of the group~$P^X$ in this subspace to the whole group~$U(p,q)^X$.

Fix the following vector:
\begin{equation}
Q(\omega)=\otimes_{(s,x)\in\omega} f(s,x), \quad\text{where}\quad f(s,x)=e^{-\frac12}|s|^2  1_S \in K_S.
\end{equation}

The formula for the characteristic functional of the quasi-Poisson measure~$\sigma$ implies that this vector lies in the space $\operatorname{QPS}(\mathcal E(Y),\sigma,K^\otimes)$ and has norm~$1$.

Denote by $b(\widetilde  g)=b(g(\,\cdot\,))$ the trivial 1-cocycle in the space~$L$ generated by the vector~$Q$, i.e.,
$$b\widetilde (g) = U(\widetilde  g) Q - Q.$$

We will denote by $L=L(\mathcal E(Y),\sigma,K^\otimes)$ the pre-Hilbert space linearly spanned by the vectors~$b(g(\cdot))$.

The construction of the desired extension literally reproduces the construction of the analogous extension of a representation of the Iwasawa group to  the group~$U(p,q)$. Namely, let $K$ be the maximal compact subgroup in~$U(p,q)$. Consider the decomposition
$$
U(p,q)^X=P^X K^X.
$$

As in the case of the group $U(p,q)$, we define the action an element of the group~$K^X$ on the set of vectors $b(\widetilde p)$, $\widetilde p \in P^X$, as a permutation:
$$
U(\widetilde k)b(\widetilde p)=b(\widetilde p'), \quad \widetilde p'\in P^X,
$$
where $\widetilde p'$ is determined by the relation $\widetilde k \widetilde p=\widetilde p'\widetilde k'$, $\widetilde k'\in K^X$.
The operators thus defined for the elements of~$K^X$ satisfy the group property, and they can be extended by linearity to operators on the whole space~$L$.

It is not difficult to check that together with the representation operators of the original group~$P^X$ they generate a representation of the whole group~$U(p,q)^X$.

Note that in the resulting representation, the operators corresponding to elements of the center of~$U(p,q)$ are identity operators.

\subsection{The involution operator~$w$}

The group of currents~$U(p,q)^X$ is algebraically generated by the set of elements of the group~$P^X$ and a single element~$w$ of the compact group~$K$; hence, to extend the representation of~$P^X$ in the space~$L$ to~$U(p,q)^X$, it suffices to describe the action of the single operator~$U(w)$. It follows from the definition that this action is given by the formula analogous to the formula for the action of the involution in the representation~$T$ of the group~$U(p,q)$.

\section*{Conclusion}

Our construction is completed: we have obtained a well-defined (nonunitary) operator-irreducible representation of the group of measurable currents with values in the group~$U(p,q)$ for arbitrary positive integers
$p,q$ with $q\geq p$ in a quasi-Poisson pre-Hilbert space, where the operators corresponding to the subgroup of currents of the Iwasawa group act by bounded (everywhere defined) operators, and the involution is defined, in general, only on a dense subspace.

The further investigation of the properties of this and similar representations is a matter for the future. One can expect that this investigation will lead to a considerable extension of tools and possibilities of representation theory.

\end{document}